\documentclass[sn-mathphys]{sn-jnl}
\usepackage{amsthm}
\usepackage{amsfonts}
\usepackage{amssymb}
\usepackage{enumerate}
\usepackage{thmtools}
\usepackage{cleveref}
\catcode`\|=12
\usepackage{tikz}
\usetikzlibrary{babel}
\newtheorem{theorem}{Theorem}
\newtheorem{lemma}{Lemma}
\newtheoremstyle{sltheorem}
{2pt}                
{2pt}                
{\slshape}        
{}                
{\slshape}       
{:}               
{ }               
{ }                
\theoremstyle{sltheorem}
\newtheorem{claim}{Claim}
\newcommand{\newcategory}[1]{\expandafter\newcommand\csname #1\endcsname{\mathbf{#1}}}

\begin{document}
\title{Correction: N-free posets and orthomodularity}
\author{Gejza Jenča}\email{gejza.jenca@stuba.sk}
\affil{\orgdiv{Department of Mathematics and Descriptive Geometry, Faculty of Civil Engineering}, \orgname{Slovak University of Technology},
\orgaddress{\street{Radlinského 11},
\city{Bratislava},
\postcode{810 05},
\country{Slovakia}}}
\abstract{
We present a corrected version of a theorem from the paper \emph{N-free posets and
orthomodularity} published in  Order 43(1) (2026).
}

\maketitle

\section{Introduction}

Regrettably, in \cite[Theorem 4.3]{jenca2026nfree} it has been claimed by the author of this note that a finite poset $P$ has a
compatible incomparability orthoset if and only if $P$ does not contain a certain
type of configuration of elements, which we called ,,weak N''.  This is not true. 
The corrected version, presented here, states that a finite poset $P$ has a compatible
incomparability orthoset iff there is no weak N and no X in $P$. 

\section{The counterexample}
Let $P$ be a poset. 
We say that $(a,b,c,d)\in P^4$ form \emph{a weak N}, if
$a<c$, $b\prec c$, $b<d$, $a\perp b$,  $c\perp d$. We denote this by
$N^?(a,b,c,d)$.

Let us write $a\perp b$ to express the fact that two elements of a poset
are incomparable and $x^\perp$ for the set of all elements incomparable to $x$.
We say that a finite poset $P$ \emph{has a compatible incomparability orthoset}
if for all $x,y\in P$ such that $x\not\perp y$, there exists $z\in P$ such that
$x^\perp\cup y^\perp\subseteq z^\perp$.

\begin{figure}
\begin{center}
\begin{tikzpicture}[x=3em,y=3em]

\node[circle, fill=black, draw=black, inner sep=0pt, minimum size=5pt, label={[label distance=2.5pt]-90:{\small $a$}}] (a) at (0,0) {};
\node[circle, fill=black, draw=black, inner sep=0pt, minimum size=5pt, label={[label distance=2.5pt]-90:{\small $b$}}] (b) at (1,0) {};
\node[circle, fill=black, draw=black, inner sep=0pt, minimum size=5pt, label={[label distance=2.5pt]180:{\small $m$}}] (m) at (0.5,0.5) {};
\node[circle, fill=black, draw=black, inner sep=0pt, minimum size=5pt, label={[label distance=2.5pt]90:{\small $c$}}] (c) at (0,1) {};
\node[circle, fill=black, draw=black, inner sep=0pt, minimum size=5pt, label={[label distance=2.5pt]90:{\small $d$}}] (d) at (1,1) {};

\draw[-, color=black, line width=0.7pt, solid] (a) -- (m);
\draw[-, color=black, line width=0.7pt, solid] (m) -- (c);
\draw[-, color=black, line width=0.7pt, solid] (b) -- (m);
\draw[-, color=black, line width=0.7pt, solid] (m) -- (d);

\end{tikzpicture}
\end{center}
\caption{The X}
\label{fig:X}
\end{figure}
In \cite[Theorem 4.3]{jenca2026nfree} the author of this note stated that a finite poset $P$ has a
compatible incomparability orthoset if and only if $P$ does not contain a weak N.
This is not true, the counterexample is the poset in Figure \ref{fig:X}. 
Indeed, there is no weak N in that poset. 
Observe that $a\not\perp c$, $a^\perp\cup c^\perp=\{b,d\}$, but there
is no element that would be incomparable to both $b$ and $d$. It turns out that
besides weak N, this type of configuration of elements is another one we must forbid.

\section{The correction}

We say that $(a,b,c,d,m)\in P^5$ form \emph{an X}, if
$a<m<c$, $b<m<d$, $a\perp b$ and $c\perp d$. We denote this
by $X(a,b,c,d;m)$. A \emph{covering X} is an $X(a,b,c,d;m)$ with
$m\prec c$. We denote this by $X^\prec(a,b,c,d;m)$.
\begin{lemma}\label{lemma:XimpliesCoveringX}
A finite poset contains an X if and only if it contains a covering X.
\end{lemma}
\begin{proof}
Consider the set
\[
M=\{x\in P:m\leq x\text{ and }x<c\text{ and }x< d\}
\]
If $m'$ is a maximal element of $M$, then $X(a,b,c,d;m')$.
Consider now the set
\[
A=\{x\in P:m'<x\leq c\}
\]
As $m'$ is a maximal common lower bound of $c,d$, we have
$A\perp d$. If $c'$ is a minimal element of $A$, then
$m'\prec c'$ and we see that $X^\prec(a,b,c',d;m')$.
\end{proof}

\begin{theorem}
For every finite poset $P$, the following are equivalent.
\begin{enumerate}[(a)]
\item There is no weak N and no X in $P$.
\item $P$ has a compatible incomparability orthoset.
\end{enumerate}
\end{theorem}
\begin{proof}

(b)$\implies$(a): 
Let us prove that if $P$ contains a weak $N$, then $(P,\perp)$ is not compatible.
Suppose that we have $N^?(a,b,c,d)$ in $P$. As $b\prec c$, $c\not\perp b$. We will prove that there is
no $z\in P$ such that $c^\perp\cup b^\perp\subseteq z^\perp$. 

Assume the contrary; this implies that $z\not\perp b$ and $z\not\perp c$.
Indeed, $z\perp b$ would mean that $z\in b^\perp\subseteq z^\perp$ and hence $z\perp z$ 
which is not true, so $z\not\perp b$. Similarly $z\not\perp c$. Furthermore, as $a\in b^\perp$
and $d\in c^\perp$, we see that $\{a,d\}\subseteq b^\perp\cup c^\perp\subseteq z^\perp$, so
$z\perp a$ and $z\perp d$.

We have proved that $z\not\perp b$. This is equivalent to $z\leq b$ or $b<z$. However,
$z\leq b<d$ contradicts $z\perp d$, so $b<z$. Similarly, we can prove that $z<c$. But $b<z<c$
contradicts with our assumption that $b\prec c$.

Suppose now that $P$ contains an X. By \Cref{lemma:XimpliesCoveringX}, 
$P$ contains a covering X, say $X^\prec(a,b,c,d;m)$. Assume that $(P,\perp)$ is 
compatible. As $a\not\perp c$, there is $z\in P$ such that 
$a^\perp\cup c^\perp\subseteq z^\perp$. If $a\perp z$, then $z\in a^\perp$, but then
$z\perp z$ which is not true, hence we see that $a\not\perp z$. Similarly,
$c\not\perp z$. Now, $c\not\perp z$ implies either $c\leq z$ or $z<c$. But if $c\leq
z$, then $b<z$, which contradicts $b\in a^\perp\subseteq z^\perp$. Hence
$z<c$. Similarly, $a<z$. If $z\not\perp m$, then $z<m<d$, which contradicts 
$d\in c^\perp\subseteq z^\perp$ or $b<m\leq z$, which contradicts $b\in
a^\perp\subseteq z^\perp$. Therefore, $z\perp m$.

Summarizing, we have $z<c\succ m<d$, $c\perp d$ and $z\perp m$, hence
$N^?(z,m,c,d)$ and hence by the previous part of the proof, 
$P$ is not compatible -- a contradiction.

\noindent(a)$\implies$(b): Suppose that $(P,\perp)$ is not compatible. We will
prove that there is some weak N in $P$, or some X in $P$.

Since $(P,\perp)$ is not compatible, there are $x,y\in P$ with $x\not\perp y$ 
such that there does not exist $z\in P$ such that $x^\perp\cup y^\perp\subseteq z^\perp$.
Equivalently, there are $x,y\in P$ with $x\not\perp y$, such that for all $z\in P$, $x^\perp\cup y^\perp\not\subseteq z^\perp$.

Clearly, $x\neq y$, because otherwise we could put $z=x$. Hence $x\not\perp y$
implies that either $x<y$ or $x>y$. Without loss of generality, we may assume
that $x<y$.  If $x^\perp\subseteq y^\perp$, then $x^\perp\cup y^\perp=y^\perp$
and we may put $z=y$, contradicting our assumptions. So $x^\perp\not\subseteq
y^\perp$ and, similarly, $y^\perp\not\subseteq x^\perp$. Hence there must be
$a\in x^\perp\setminus y^\perp$ and $d\in y^\perp\setminus x^\perp$.
Summarizing, we see that $a\perp x$, $a\not\perp y$, $d\perp y$ and $d\not\perp x$.

As $a\not\perp y$, either $a<y$ or $y\leq a$. If $y\leq a$, then $x<y\leq a$, but
this contradicts $a\perp x$. So $a<y$ and, similarly, we can prove that $x<d$. If
$x\prec y$, then $N^?(a,x,y,d)$ and we are done. If $x\not\prec y$, then there is
$z\in P$ be such that $x<z<y$. By assumption, $(P,\perp)$ is not compatible, hence it is not true that $x^\perp\cup
y^\perp\subseteq z^\perp$. That means, there is $w\in x^\perp\cup y^\perp$ such that
$w\not\perp z$. Since $z\not\in x^\perp\cup y^\perp$, $w\neq z$. Hence, for every
$z\in P$ with $x<z<y$ there is $w\in x^\perp\cup y^\perp$ such that either $w<z$ or
$z<w$. However, if $w<z$, then $w<z<y$ implies $w\notin y^\perp$, so we must have
$w\in x^\perp$. Similarly $z<w$ implies $w\in y^\perp$. Summarizing, we see that
we have proved the following claim.
\begin{claim}
For every $z\in P$ with $x<z<y$ there is $w\in P$ such that exactly
one of the following is true:
\begin{itemize}
\item $w<z$ and $w\in x^\perp$
\item $z<w$ and $w\in y^\perp$
\end{itemize}
\end{claim}
With this property of every such $z$ in mind, let us consider a maximal chain 
$$
x=z_0\prec z_1\prec\dots\prec z_n\prec z_{n+1}=y
$$
in the closed interval $[x,y]$ of $P$. There is a sequence 
\[
d=w_0,w_1,\dots,w_n,w_{n+1}=a
\]
such that for each $i$ either $x\perp w_i<z_i$ or $z_i<w_i\perp y$. Let $j$ be
the maximal index such that $z_j<w_j\perp y$.
As $y=z_{n+1}\not <w_{n+1}=a$, we see that $j<n+1$. 
Moreover, by the maximality of $j$, we may conclude that 
$x\perp w_{j+1}<z_{j+1}$. Let us write down what we know about all of these
elements so far.
\begin{itemize}
\item $w_{j+1}<z_{j+1}$
\item $z_j\prec z_{j+1}$
\item $z_j<w_j$.
\end{itemize}
If, in addition to these facts, $z_{j+1}\perp w_j$ and $w_{j+1}\perp z_j$, then 
$N^?(w_{j+1},z_j,z_{j+1},w_j)$ and we have found a weak N in $P$.

We will now prove that if either $z_{j+1}\not\perp w_j$ or
$w_{j+1}\not\perp z_j$, then there is an X in $P$.

Suppose that $z_{j+1}\not\perp w_j$, so that
either $w_j<z_{j+1}$ or $z_{j+1}=w_j$ or $z_{j+1}<w_j$.
If $w_j<z_{j+1}$, then $z_j<w_j<z_{j+1}$, contradicting $z_j\prec z_{j+1}$.
If $z_{j+1}=w_j$, then $z_{j+1}\perp y$, contradicting $z_{j+1}\leq y$. Therefore,
$z_{j+1}<w_j$. If $j=n$, then $z_{j+1}=y<w_j$ contradicts $w_j\perp y$. Hence
$j<n$ and thus $z_{j+1}<z_{n+1}=y$. Summarizing, under the assumption
$z_{j+1}\not\perp w_j$ we obtain that
\begin{itemize}
\item $w_{j+1}\perp x$ by the maximality of $j$;
\item $w_{j+1}<z_{j+1}$ by the required property of $w_{j+1}$;
\item $x<z_{j+1}$, because $x=z_0$;
\item $z_{j+1}<w_j$, by previous paragraph;
\item $z_{j+1}<y$, by previous paragraph;
\item $w_j\perp y$ by the required property of $w_j$.
\end{itemize}
But this exactly means that $X(w_{j+1},x,y,w_j;z_{j+1})$ and we have found an X in
$P$.

Similarly, the assumption $w_{j+1}\not\perp z_j$ implies that there is an X in $P$,
as this situation is essentially dual to the $z_{j+1}\not\perp w_j$ case. 
Specifically: $w_{j+1}\not\perp z_j$ forces $w_{j+1} < z_j$, since $z_j < w_{j+1}$ would violate $z_j\prec z_{j+1}$ (as $w_{j+1} < z_{j+1}$), and $w_{j+1} = z_j$ contradicts $w_{j+1}\perp x$ (since $z_j > x$).
Note also that $j\geq 1$ in this case: if $j = 0$, then $z_j = x$ and $w_{j+1}\not\perp x$ would contradict $w_{j+1}\perp x$.

The resulting configuration is $X(x, w_{j+1}, w_j, y;\, z_j)$, whose conditions are:
$x < z_j < w_j$,
$w_{j+1} < z_j < y$
(using $w_{j+1} < z_j$ from above and $z_j < z_{j+1} \leq y$),
$x\perp w_{j+1}$, and $w_j\perp y$ -- all verified.

\end{proof}
\section{Supplementary material}

All the results from this note and all valid results from
\cite{jenca2026nfree} have been formalized in Lean 4. The source code is available
in the repository at the URL
\url{https://github.com/gjenca/n-free-posets-and-orthomodularity.git}.

\section*{Declarations}
\begin{description}
\item[\bf Funding:]
This research is supported by grants VEGA 2/0128/24 and 1/0036/23,
Slovakia.
\item[\bf Conflict of interest/Competing interest:]None.
\item[\bf Availability of data and material:]Not applicable.
\item[\bf Code availability:]Not applicable.
\item[\bf Authors' contributions:]There is a single author. 
\end{description}


\
\end{document}